\documentclass[leqno]{amsart}

\usepackage{amscd,amssymb,amsmath}

\usepackage[bookmarks]{hyperref}
\hypersetup{backref, colorlinks=true}

 \usepackage{graphicx}
\usepackage[all]{xy}
\newtheorem{theorem}{Theorem}[section]
\newtheorem{lemma}[theorem]{Lemma}
\newtheorem{proposition}[theorem]{Proposition}
\newtheorem{corollary}[theorem]{Corollary}
\newtheorem{remark}[theorem]{Remark}

\newtheorem{example}[theorem]{Example}
\newtheorem{definition}[theorem]{Definition}

\newenvironment{equationth}{\stepcounter{theorem}\begin{equation}}{\end{equation}}

\def\C{ \mathbb{C}}

\def\R{ \mathbb{R}}

\def\N{ \mathbb{N}}

\def\S{ \mathbb{S}}

\def\codim{ {\rm codim}}
\def\rang{ {\rm rank}}
\def\corang{ {\rm corank}}
\def\Reg{ {\rm Reg}}

\def\rond{\mathaccent"7017}

\begin{document}

\title[]{\small ON A SINGULAR VARIETY ASSOCIATED TO A POLYNOMIAL MAPPING}
\makeatother

\author[Nguyen Thi Bich Thuy, Anna Valette and Guillaume Valette]{Nguyen Thi Bich Thuy, Anna Valette and Guillaume Valette}

\address[Nguyen Thi Bich Thuy]{I.M.L., Aix-Marseille University,  Case 907, Luminy 13288-Marseille Cedex 9, France}
\email{ntbthuy@ctu.edu.vn}
\address[Anna Valette]{Instytut Matematyki Uniwersytetu
Jagiello\'nskiego, ul. S \L ojasiewicza, Krak\'ow, Poland}
\email{anna.valette@im.uj.edu.pl}
\address[Guillaume Valette]
{Instytut Matematyczny PAN, ul. \'Sw. Tomasza 30, 31-027 Krak\'ow,
Poland} \email{gvalette@impan.pl}


\maketitle \thispagestyle{empty}

\begin{abstract} In the paper "{\it Geometry of polynomial mapping at infinity via intersection homology}" the second and third authors associated to a given polynomial mapping $F : \C^2 \to \C^2$ with nonvanishing jacobian a variety whose homology or intersection homology describes the geometry of singularities at infinity of the mapping. We generalize this result.
\end{abstract}

\section{INTRODUCTION}
In 1939, O. H. Keller \cite{kel} stated the famous conjecture which is called the Jacobian conjecture: any polynomial mapping $F: \C^n\to\C^n$ with nowhere vanishing Jacobian is a polynomial automorphism. The problem remains open today even for dimension 2.  We call the smallest set $S_F$ such that the mapping $F:X\setminus F^{-1}(S_F)\to Y\setminus S_F $ is proper, the asymptotic set of $F$. The Jacobian conjecture  reduces to show that the asymptotic set of  a complex  polynomial  mapping with nonzero constant Jacobian is empty.  So the set of points at which a polynomial mapping fails to be proper plays an important role.

The second and third authors gave in \cite{Valette}  a new approach to study the Jacobian conjecture in the case of dimension 2: they constructed a real pseudomanifold denoted $N_F \subset \R^{\nu}$, where $\nu > 2n$, associated to a given polynomial mapping $F : \C^n \to \C^n$, such that 
the singular part of the variety $N_F$ is contained in $(S_F \times K_0(F)) \times \{ 0_{\R^{\nu -2n}}\}$ where $K_0(F)$ is the set of critical values of $F$. 
 In the case of dimension 2, the homology or intersection homology of $N_F$ describes the geometry of the  singularities at infinity of 
 the mapping $F$.

Our aim is to improve this result 
in the general case of dimension $n >2$ and compute the intersection homology of the associated pseudomanifold $N_F$. 
Let $\hat {F_i}$ be the leading forms of the components $F_i$ of the polynomial mapping $F = (F_1, \ldots, F_n) : \C^n \rightarrow \C^n$. 
We show (Theorem \ref{valettethuy2}) that if the  polynomial mapping $F : \C^n \rightarrow \C^n$  has nowhere vanishing Jacobian and if {$rank {(D \hat {F_i})}_{i=1,\ldots,n} > n-2$},  then the condition of properness  of $F$ 
is equivalent to the condition of vanishing  homology or intersection homology of $N_F$. Moreover, it 
is  indeed more precise to compute   intersection homology than homology. 
In order to compute the intersection homology of the variety $N_F$, we show that it admits a stratification which is  locally topologically trivial along the strata. The main feature of intersection homology is to satisfy Poincar\'e duality that is more interesting in the case where the stratification has no stratum of odd real dimension. We show that the variety $N_F$ 
admits a Whitney stratification with only even dimensional strata.  It is well known that Whitney stratification are locally topologically trivial along the strata. 


\subsection*{Acknowledgment}  It is our pleasure to thank J.-P. Brasselet for his interest and encouragements. 

\section{PRELIMINARIES AND BASIC DEFINITION}
In this section we set-up  our framework. All the considered sets in this article are semi-algebraic.

\subsection{Notations and conventions.}\label{section_notations}
Given a topological space $A$, singular simplices of $A$ will be semi-algebraic continuous mappings $\sigma:T_i \to A$, 
where $T_i$ is the standard $i$-simplex in $\R^{i+1}$. 
Given a subset $X$ of $\R^n$ we denote by $C_i(X)$ the group of
$i$-dimensional singular chains (linear combinations of singular simplices
with coefficients in $\R$); if
$c$ is an element of $C_i(X)$, we denote by $|c|$ its support.
By $Reg(X)$ and $Sing(X)$ we denote respectively the regular and singular locus of the set $X$. Given $A \subset \R^n$, $\overline{A}$ will stand for the topological closure of $A$. Given a point $x\in \R^n$ and $\alpha>0$, we write $\mathbb{B}(x,\alpha)$ for the ball of radius $\alpha$ centered at $x$ and $\S(x,\alpha)$ for the corresponding sphere, boundary of $\mathbb{B}(x,\alpha)$.

\subsection{Intersection homology.}
We briefly recall the definition of intersection homology; for details,
we refer  to the fundamental work of M. Goresky and R. MacPherson
\cite{GM1} (see also \cite{Jean Paul}).

\begin{definition} 
{\rm Let $X$ be a $m$-dimensional semi-algebraic set.  A {\bf semi-algebraic stratification of $X$} is the data of a finite semi-algebraic filtration 
\begin{equationth} \label{dfn_stratification}
X = X_{m} \supset X_{m-1} \supset \cdots \supset X_0 \supset X_{-1} = \emptyset
\end{equationth}
 such that  for every $i$,  the set $S_i = X_i\setminus X_{i-1}$ is either empty or a topological manifold of dimension $i$. A connected component of $S_i$ is  called   {\bf a stratum} of $X$.}
\end{definition}

\begin{definition} [\cite{Whitney}]
{\rm One says that the {\bf Whitney $(b)$ condition} is realized for a stratification
 if for each pair of strata $(S, S')$ 
 and for any $y \in S$ one has: Let $\{ x_n\}$ be a sequence of points in $S'$ with limit $y$ 
and let $\{ y_n \}$ be a sequence of points in $S$ tending to $y$, 
assume that the sequence of tangent spaces $\{ T_{x_n}S' \}$ admits a limit $T$ for $n$ tending to $+ \infty$ (in a suitable Grassmanian manifold)
and that the sequence of directions $x_ny _n$ admits a limit $\lambda$ for $n$ tending to $+ \infty$ (in the corresponding projective manifold), then $\lambda \in T$. 
}
\end{definition}

We denote by $cL$  the open cone on the space $L$, the cone on the empty set being a point. Observe that if $L$ is a stratified set then $cL$ is stratified by  the cones over the strata of $L$ and a $0$-dimensional stratum (the vertex of the cone). 

\begin{definition}
{\rm A stratification of $X$ is said to be {\bf locally topologically trivial} if for every $x \in X_i\setminus X_{i-1}$, $i \ge 0$, there is an open neighborhood $U_x$  of $x$ in $X$, a stratified set $L$ and a semi-algebraic homeomorphism 
 $$h:U_x \to (0;1)^i \times  cL,$$   such  that
 $h$ maps the strata of $U_x$ (induced stratification) onto the strata of  $  (0;1)^i \times cL$  (product stratification).}
\end{definition}

\begin{definition} \label{del pseudomanifolds}
{\rm A  {\bf pseudomanifold} in $\R^n$ is a  subset $X \subset \R^n$ whose singular locus
is of codimension at least 2 in $X$ and whose regular locus is dense in $X$.

A {\bf stratified pseudomanifold} (of dimension $m$) is the data of an $m$-dimensional pseudomanifold $X$ together with  a semi-algebraic filtration: $$ X = X_m \supset X_{m -1} \supset X_{m -2} \supset \cdots  \supset X_0 \supset X_{-1} = \emptyset,$$  which constitutes a locally topologically trivial stratification of $X$.}
\end{definition}

\begin{definition} \label{dfn_pseudo_a_bord}
{\rm A  {\bf stratified  pseudomanifold with boundary} is a semi-algebraic couple 
$(X,\partial X)$ together with a semi-algebraic  filtration $$ X = X_m \supset X_{m -1} \supset X_{m -2} \supset \cdots  \supset X_0 \supset X_{-1} = \emptyset,$$ 
such that:
\begin{enumerate}
\item  $X\setminus \partial X$ is an  $m$-dimensional stratified pseudomanifold (with the filtration $X_{j}\setminus \partial X$),
\item  $\partial X$ is a  stratified pseudomanifold (with the filtration $X_j':=X_{j+1}\cap \partial X$),
\item $\partial X$ has a  {\bf stratified collared neighborhood}: there exist a neighborhood $U$ of $\partial X$ in $X$ and a semi-algebraic homeomorphism $\phi:  \partial X \times [0,1]\to U $ such that $\phi(X_{j-1}'\times [0,1])=U \cap X_j$ and $\phi(\partial X \times \{0\})=\partial X$.
\end{enumerate}
}
\end{definition}

\begin{definition}
{\rm 
A {\bf perversity} is an $(m-1)$-uple  of integers $\bar p = (p_2,
p_3,\dots , p_m)$ such that $p_2 = 0$ and $p_{k+1}\in\{p_k, p_k + 1\}$.

Traditionally we denote the zero perversity by
$\overline{0}=(0,\dots,0)$, the maximal perversity by
$\overline{t}=(0,1,\dots,m-2)$, and the middle perversities by
$\overline{m}=(0,0,1,1,\dots, [\frac{m-2}{2}])$ (lower middle) and
$\overline{n}=(0,1,1,2,2,\dots ,[\frac{m-1}{2}])$ (upper middle). We say that the
perversities $\overline{p}$ and $\overline{q}$ are {\bf complementary} if $\overline{p}+\overline{q}=\overline{t}$.

Given a stratified pseudomanifold $X$, we say that a semi-algebraic subset $Y\subset X$ is  $(\bar
p, i)$-{\bf allowable} if  $\dim (Y \cap X_{m-k}) \leq i - k + p_k$ for
all $k \geq 2$.

In particular,  a subset $Y\subset X$ is  $(\overline{t},i)$-allowable if $\dim( Y \cap Sing(X)) <i-1$.

Define $IC_i ^{\overline{p}}(X)$ to be the $\R$-vector subspace of $C_i(X)$
consisting of those chains $\xi$ such that $|\xi|$ is
$(\overline{p}, i)$-allowable and $|\partial \xi|$ is
$(\overline{p}, i - 1)$-allowable.}
\end{definition}

\begin{definition} 
{\rm The {\bf $i^{th}$ intersection homology group with perversity $\overline{p}$}, denoted by
$IH_i ^{\overline{p}}(X)$, is the $i^{th}$ homology group of the
chain complex $IC^{\overline{p}}_*(X).$
}
\end{definition}

 Goresky and MacPherson proved that these groups are independent of the
choice of the stratification and are finitely generated \cite{GM1, GM2}.

 \begin{theorem}[Goresky,
MacPherson \cite{GM1}] {\it For any orientable compact stratified
semi-algebraic $m$-dimensional pseudomanifold  $X$,  generalized Poincar\'e duality holds:
\begin{equation}\label{eq poincare duality} IH_k ^{\overline{p}}(X)  \simeq IH_{m-k} ^{\overline{q}} (X), 
\end{equation}
where $\overline{p}$ and $\overline{q}$ are complementary perversities.}
\end{theorem}

In the non-compact case the above isomorphism holds
for Borel-Moore homology:
\begin{equation}\label{eq poincare duality} IH_k ^{\overline{p}}(X)  \simeq IH_{m-k,BM} ^{\overline{q}} (X), \end{equation}
where $IH_{*,BM}$ denotes the intersection homology with
respect to Borel-Moore chains \cite{GM2, Jean Paul}. A relative version is also true in the case where $X$ has boundary.

\medskip 

\begin{proposition}[Topological invariance, \cite{GM1, GM2}] \label{proinvariance}
{\it Let X be a locally compact stratified pseudomanifold and $\overline{p}$ a perversity, then the intersection homology groups $IH_*^{\overline{p}}(X)$ and $IH_{*,BM} ^{\overline{p}} (X)$ do not depend on the stratification of $X$.}
\end{proposition}
\subsection{$\mathcal{L}^\infty$ cohomology}
 Let $M\subset \R^n$ be a smooth submanifold.

\medskip

\begin{definition}  {\rm We say that a differential form  $\omega$ on $M$
is $\mathcal{L}^\infty$ if there exists a constant $K$ such that  for any $x
\in M$: $$|\omega(x)| \leq K. $$
 We denote by $\Omega^j_\infty
(M)$ the cochain complex constituted by all the $j$-forms
$\omega$ such that $\omega $ and $d\omega$ are both $\mathcal{L}^\infty$. The cohomology groups of this cochain complex are called the {\bf
$\mathcal{L}^\infty$-cohomology groups of $M$} and will be denoted by
$H^* _\infty(M)$.}
\end{definition}

The third author showed that   the $\mathcal{L}^\infty$
cohomology of the differential forms defined on the regular part of a pseudomanifold $X$ coincides with the intersection
cohomology of $X$ in the maximal perversity (\cite{Valette2}, Theorem 1.2.2):
\begin{theorem}\label{thm_intro_linfty} {\it Let $X$ be a compact subanalytic pseudomanifold (possibly with boundary). Then, for any $j$:
$$H_\infty ^j(Reg(X)) \simeq IH_j^{\bar t} (X).$$
Furthermore, the isomorphism is induced by the natural mapping
provided by integration on allowable simplices.}
\end{theorem}

\subsection{The Jelonek set.} Let $F : \C^n \to \C^n$ be a polynomial mapping. We denote by $S_F$ the set of points at which the mapping $F$ is not proper, {\it i.e.} 
$$S_F = \{ y \in \C^n \text{ such that } \exists \{ x_k\} \subset \C^n, \vert x_k \vert \to \infty, F(x_k) \to y\},$$
and call it the {\bf asymptotic variety} or {\bf Jelonek set} of $F$. The geometry of this set was studied by Jelonek in a series of papers \cite{Jelonek1, Jelonek2, Jelonek3}. Jelonek obtained a nice description of this set and gave an upper bound for its degree. For the details and applications of these results we refer to the works of Jelonek. In our paper, we will need the following powerful theorems.

\begin{theorem} [\cite{Jelonek1}] \label{Jelonek1} {\it If $F: \C^n \to \C^n$ is a generically finite polynomial mapping, then $S_F$ is either an $(n-1)$ pure dimensional $\C$-uniruled algebraic variety or the empty set.}

\end{theorem}

\begin{theorem} [\cite{Jelonek1}] \label{Jelonek2} {\it If $F : X \to Y$ is a dominant polynomial map of smooth affine varieties of the same dimension then $S_F$ is either empty or is a hypersurface.}
\end{theorem}


Here, by a $\C$-uniruled variety $X$ we mean that through any point of $X$ passes a rational complex curve included in $X$. In other words, $X$ is $\C$-uniruled if for all $x \in X$ there exists a non-constant polynomial mapping $\varphi_x: \C \to X$ such that $\varphi_x(0) = x$.

In the real case, the Jelonek set is an $\R$-uniruled semi-algebraic set but, if nonempty, its dimension can be any integer between 1 and $(n-1)$   \cite{Jelonek3}.

\section{THE VARIETY $N_F$} 
The variety $N_F$ was constructed by the second and third authors in \cite{Valette}. Let us recall briefly this construction.
\subsection{Construction of the variety $N_F$ (\cite{Valette})} \label{constructionNF}
We will consider polynomial mappings $F : \C^n \to \C^n$ as real ones $F : \R^{2n} \to \R^{2n}$. By $Sing(F)$ we mean the singular locus of $F$ that   is the zero set of its Jacobian determinant and we denote by $K_0(F)$ the set of critical values of $F$,
 {\it  i.e.}  the set $F(Sing(F))$. 

We denote by $\rho$ the Euclidean Riemannian metric in $\R^{2n}$. We can pull it back in a natural way:
$$F^* \rho_x (u,v) : = \rho (d_x F(u), d_x F(v)).$$ 
Define the Riemannian manifold  $M_F : = (\R^{2n} \setminus Sing(F); F^{*} \rho)$ and observe that the mapping $F$ induces a local isometry nearly any point of $M_F$.

\begin{lemma} [\cite{Valette}] \label{Valette1} 
{\it There exists a finite covering of $M_F$ by open semi-algebraic subsets such that on every element of this covering, the mapping $F$ induces a diffeomorphism onto its image.}
\end{lemma}

\begin{proposition} [\cite{Valette}]  \label{valette2} 
{\it Let $F: \C^n \to \C^n$ be a polynomial mapping. There exists a real semi-algebraic pseudomanifold $N_F \subset \R^{\nu}$,
 for some $\nu  = 2n + p$, where $p>0$  such that 
$$ Sing (N_F) \subset (S_F \cup K_0(F)) \times \{0_{\R^{p}}\},$$
and there exists a semi-algebraic bi-Lipschitz mapping
$$h_F: M_F \to N_F \setminus (S_F \cup K_0(F)) \times \{0_{\R^{p}}\}$$
where $N_F \setminus (S_F \cup K_0(F)) \times \{0_{\R^{p}}\}$ is equipped with the Riemannian metric induced by $\R^{\nu}.$}
\end{proposition}

The variety $N_F$ is constructed as follows: let $F: \C^n \to \C^n$ be a polynomial mapping. Thanks to 
Lemma \ref{Valette1}, there exists a covering $\{ U_1, \ldots , U_p \}$ of $M_F = \R^{2n} \setminus Sing(F)$ by open semi-algebraic subsets (in $\R^{2n}$) such that on every element of this covering, the mapping $F$ induces a diffeomorphism onto its image. We may find some semi-algebraic closed  subsets $V_i \subset U_i$ (in $M_F$) which cover $M_F$ as well. Thanks to Mostowski's Separation Lemma (see Separation Lemma in \cite{Mos}, page 246), for each $i$,  $\, i =1, \ldots , p$, there exists a Nash function $\psi_i : M_F \to \R$,  such that  $\psi_i$ is positive on $V_i$ and negative on $M_F \setminus U_i$. We define 
$$h_F : = (F, \psi_1, \ldots, \psi_p) \text{ and } N_F : = \overline{h_F(M_F)}.$$

In order to prove $h_F$ is bi-Lipschitz, we do as follows: choose $x \in M_F$, then there exists $U_j$ such that $x \in U_j$ and  the mapping $F_{|U_j}:U_j \to \R^{2n} $ is a diffeomorphism onto
its image.  Define, for $y \in F(U_j)$, the following functions:
\begin{equationth}
\tilde{\psi}_i(y):=\psi_i \circ (F_{|U_j})^{-1}(y),
\end{equationth}
for $i=1,\dots ,p$, and
\begin{equationth}
\hat{\psi}(y):=(y,\tilde{\psi_1}(y),\dots,\tilde{\psi_p}(y)).
\end{equationth}
We then have the formula
\begin{equationth}h_F(x)=(F(x),\tilde{\psi}_1(F(x)),\dots,\tilde{\psi}_p(F(x)))=\hat{\psi} (F(x)).\end{equationth}
As the map $F: (U_j,F^\ast\rho) \to F(U_j)$ is bi-Lipschitz, it is enough to
show that $\hat{\psi}:F(U_j) \to \R^{2n+p}$ is
bi-Lipschitz. This amounts to prove that $\tilde{\psi}_i$ has
bounded derivatives for any $i=1,\dots,p$. In order to prove this, we chose the functions $\psi_i$ sufficiently small, by using  \L ojasiewicz inequality in the following form:
\begin{proposition} \cite{bcr} \label{lojasiewicz}
Let $A \subset \R^n$ be a closed semi-algebraic set and $f:A \to\R$ a continuous semi-algebraic function. There exist
$c\in\R, c \geq 0$ and $q \in\N$ such that for any $x\in A$ we have
$$|f(x)|\le c(1+|x|^2)^q.$$
\end{proposition} 

In fact, we can choose the Nash functions $\psi_i$ sufficiently small by multiplying $\psi_i$ by a huge power
of $\frac{1}{1+|x|^{2}}$ which is a Nash function (see Proposition 2.3 in \cite{Valette}).

Thanks to  \L ojasiewicz inequality, we also can choose the functions $\psi_i$ such that they tend to zero at infinity and near $Sing (F)$. This is the reason why the singular part of $N_F$ is contained in  $(S_F \cup K_0(F)) \times \{0_{\R^{p}}\}$. 

Moreover, the following diagram is commutative:

\begin{equationth} \label{dia1}
\xymatrix{
M_F \ar[dr]^{F}  \ar[r]^{h_F}  &     N_F    \ar[d]^{\pi_F}   \\
           &  \R^{2n} \setminus K_0(F),}
\end{equationth} 

\noindent where $\pi_F$ is the canonical projection on the first $2n$  coordinates, and $h_F$ is bijective onto its image $ N_F \setminus ((S_F \cup K_0(F)) \times \{0_{\R^{p}}\})$.

\medskip
Remark that the set  $N_F$ is not unique, it depends on the covering of $M_F$ that we choose and on the choice of the Nash function $\psi_i$.

\medskip

We see that in the complex case, even in the case $\C^2$, the real dimension of the variety $N_F$ is greater than 3, so it is difficult to draw the variety $N_F$ in this case. The natural question  arises if   the variety $N_F$ exists in the real case. The answer is yes, but we note that in this case, the variety $N_F$ is not necessarily a pseudomanifold, because in the real case, the real dimension of the Jelonek set of a polynomial mapping $F : \R^n \to \R^n$ can be $n-1$.

\begin{proposition} [\cite{Valette}, \cite{Thuy1}] \label{thuyvalette0}
{\it Let $F: \R^n \to \R^n$ be a polynomial mapping. There exist

a) a real semi-algebraic variety $N_F \subset \R^{\nu}$, for some $\nu =n+p$ 
where $ p>0$, such that 
$$ Sing (N_F) \subset (S_F \cup K_0(F)) \times \{0_{\R^{p}}\} \subset \R^{n} \times \R^{p},$$

b) a semi-algebraic bi-Lipschitz mapping
$$h_F: M_F \to N_F \setminus (S_F \cup K_0(F)) \times \{0_{\R^{p}}\}$$
where $N_F \setminus (S_F \cup K_0(F)) \times \{0_{\R^{p}}\}$ is equipped with the Riemannian metric induced by $\R^{\nu}.$}
\end{proposition}

In order to understand better the variety $N_F$, we give here an example in the real case.

\subsection{Example}

\begin{example} {\cite{Thuy1}} \label{exThuy1}
{\rm Let $F: \R^2_{(x,y)} \to \R^2_{(\alpha, \beta)}$ be the polynomial mapping defined by
$$F(x,y) = (x, x^2y(y+2)).$$

Let us construct the variety $N_F$ in this case. By an easy computation, we find:

 $\qquad \qquad \qquad \qquad \qquad Sing(F) = \{ (x, y) \in \R^2_{(x, y)} : x = 0 \text{ or } y = -1 \},$ 

$\qquad \qquad \qquad \qquad \qquad K_0(F) = \{ (\alpha, \beta) \in \R^2_{(\alpha, \beta)} : \beta = -\alpha^2 \},$

$\qquad \qquad \qquad \qquad \qquad S_F = \{ (0, \beta) \in \R^2_{(\alpha, \beta)} : \beta \geq 0 \}.$

We see that $\R^2$ is divided into four open subsets $U_i$ by $Sing(F)$ (see the Figure 1a). The mapping $F$ is a diffeomorphism on each $U_i$, for $i = 1, \ldots , 4$. Observe that $U_i$ is closed in $M_F$ so that we can chose $V_i = U_i$ for $i = 1, \ldots, 4$ (see section \ref{Valette1}).  There  exist  
Nash functions  $\psi_i : M_F \to {\mathbb R}$ such that each $\psi_i $ is positive on $U_i$ and  negative on $U_j$ if $j \neq i$. Since $N_F$ is the closure of $h_F(M_F)$ where $h_F= (F, \psi_1, \ldots, \psi_4)$, then $N_F$ have 4 parts ${(N_F)}_1, \ldots, {(N_F)}_4$ where ${(N_F)}_i$ is the closure of $h_F(U_i)$ for $i = 1, \ldots, 4$.

Again, an easy computation shows:
$$ F(U_1) = F(U_2) = \{(\alpha, \beta) \in \R^2_{(\alpha, \beta) } : \alpha > 0, \beta > - \alpha^2 \},$$  
$$  F(U_3) = F(U_4) = \{(\alpha, \beta) \in \R^2_{(\alpha, \beta) }: \alpha < 0, \beta > - \alpha^2 \}.$$

Each ${(N_F)}_i$ is $F(U_i)$ embedded in $\R_{(\alpha, \beta)} \times \R^4$ but${(N_F)}_i$ does not lie  in the plane $\R_{(\alpha, \beta)}$ anymore, it is ``lifted" in $\R_{(\alpha, \beta)} \times \R^4$. However, the part contained in $K_0(F) \times S_F$ still remains in the plane $\R_{(\alpha, \beta)}$ since the functions $\psi_i$ tend to zero at infinity and near $Sing(F)$ (see the Figure 1b).

Now we want to know how the parts ${(N_F)}_i$ are glued together. Using diagram (\ref{dia1}), for any point $a = (\alpha, \beta) \in \R^2 \setminus K_0(F)$ the cardinal of  $\pi_F^{-1}(a) \setminus ((K_0(F) \cup S_F) \times \{ 0_{\R^4} \})$ is equal to the cardinal of $F^{-1}(a)$ since $h_F$ is bijective.
Consider now the equation 
$$F(x,y) = \left( x, x^2y^2 + 2x^2y \right)= (\alpha, \beta)$$
where $\beta \neq - \alpha^2.$ We have 
\begin{equationth} \label{equaex}
\alpha^2 y^2 +2\alpha^2y - \beta= 0. 
\end{equationth}
As the reduced discriminant is $\Delta'=\alpha^4 + \alpha^2\beta = \alpha^2(\alpha^2 +\beta)$, then

1) if $\beta< -\alpha^2 $, the equation (\ref{equaex}) does not have any solution,

2) if $\beta> -\alpha^2$, the equation (\ref{equaex}) has two solutions. 

Then ${(N_F)}_1$ and ${(N_F)}_2$ are glued together along $(K_0(F) \cup S_F) \times \{ 0_{\R^4} \}$.
Similarly, ${(N_F)}_3$ and ${(N_F)}_4$ are glued together along $(K_0(F) \cup S_F) \times \{ 0_{\R^4} \}$ (see the Figure 1c).

}

\begin{figure}[h!] \label{DessinEx1Thuy}
\begin{center}
\includegraphics[scale=0.6]{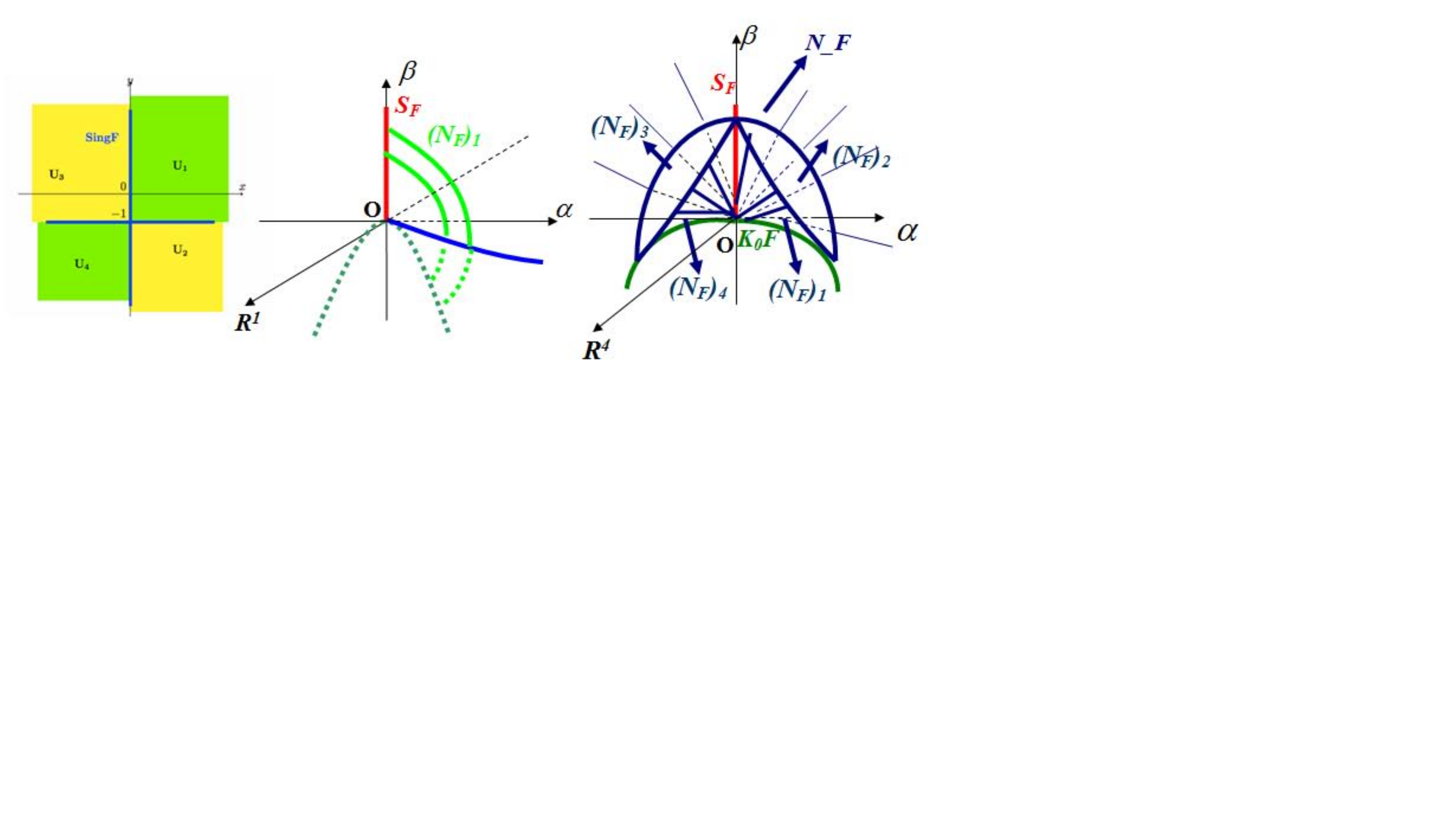}
\vskip-6cm 
\hskip -3truecm \textcolor{red}{1a}\hskip 3truecm  \textcolor{red}{1b} \hskip 4truecm\textcolor{red}{1c} 
\caption{The variety $N_F$.}
\end{center}\end{figure}
\end{example}

\medskip

\subsection{Homology and intersection homology of $N_F$}
\begin{lemma} [\cite{Valette}] \label{valette3}
{\it Let $F: \C^2 \to \C^2$ be a polynomial mapping. There exists a natural stratification of the variety $N_F$,  by even (real) dimension strata, which is locally topologically trivial along the strata.}
\end{lemma}
In fact, the stratification of the variety $N_F$ is showed in \cite{Valette} to be 
$$N_F \supset (S_F \cup K_0(F)) \times \{ 0_{\R^{p}} \} \supset (Sing(S_F \cup K_0(F)) \cup B) \times \{ 0_{\R^{p}}\} \supset \emptyset,$$
where $B = S_{F \vert_{F^{-1}(S_F)}}$. 

\begin{theorem} [\cite{Valette}] \label{valette4}
{\it Let $F: \C^2 \to \C^2$ be a polynomial mapping with nowhere vanishing Jacobian. The following conditions are equivalent:

(1) $F$ is non proper,

(2) $H_2(N_F) \neq 0$,

(3) $IH_2^{\overline{p}}(N_F) \neq 0$ for any perversity $\overline{p}$,

(4) $IH_2^{\overline{p}}(N_F) \neq 0$ for some perversity $\overline{p}$.}

\end{theorem}

We notice that, for a given polynomial map $F: \C^n \to \C^n$,
 whatever the choice of $N_F$,
 the homology and intersection homology groups $H_2(N_F)$ and
$IH_2^{\overline{p}}(N_F)$ are either $0$ or not $0$.





\section{Results}
The following theorem generalizes the Lemma \ref{valette3} and shows existence of suitable 
stratifications of the set $S_F$ in the  case of a polynomial mapping $F: \C^n \to \C^n.$

\begin{theorem} \label{valettethuy1}
{\it Let $F: \C^n \to \C^n$ be a generically finite polynomial mapping with nowhere vanishing Jacobian. There exists a filtration of $N_F$: 
$$ \qquad  N_F = V_{2n} \supset V_{2n-1}  \supset V_{2n-2}  \supset \cdots \supset  V_1 \supset V_0 \supset V_{-1} = \emptyset $$
\noindent such that :
\begin{enumerate}
\item[1)] for any  $i < n$,  $V_{2i+1} = V_{2i}$,
\item[2)] the corresponding stratification satisfies the Whitney $(b)$ condition.
\end{enumerate}  
}
\end{theorem} 


\begin{proof}[Proof]
We have the following elements 

+ Thanks to Sard Theorem, we have $\dim_{\C} Sing(S_F) \leq n-2$, {\it i.e. } $\dim _{\R} Sing(S_F)  \leq 2n - 4$. 

+ Let $M_{2n-2} = F^{-1}(S_F) \cap M_F$. The mapping $F$ restricted to  $M_{2n-2}$ is dominant.
  Thanks to Jelonek's  Theorem  (Theorem \ref{Jelonek2}), 
we have $\dim_{\C} S_{F_{\vert M_{2n-2}}} = n -2$  
 (since  $\dim_{\C} M_{2n-2} = n-1)$.  Thus, we obtain $\dim_{\C} S_{F_{\vert M_{2n-2}}} = 2n -4.$  

+ Thanks to Whitney's Theorem (Theorem 19.2, Lemma 19.3, \cite{Whitney}), the set  $B_{2n-2}$ of points $x \in S_F$ at which the Whitney $(b)$ condition fails is contained in a complex algebraic variety 
of complex dimension smaller than $n-1$, so $\dim_{\R} B_{2n-2} \leq 2n -4$.

\noindent We will define a filtration $(\mathcal{W})$ of $\R^{2n}$ by algebraic varieties and compatible with $S_F$:
$$ (\mathcal{W}) : \qquad W_{2n} = \R^{2n} \supset W_{2n-1} \supset W_{2n-2} = S_F \supset \cdots \supset W_{2k+1} \supset W_{2k} \supset \cdots \supset W_1 \supset W_0 \supset \emptyset$$ 
by decreasing induction on $k$. Assume that $W_{2k}$ has been constructed. If  $\dim_{\R} W_{2k} < 2k$ then we put
$$ W_{2k-1} = W_{2k-2} =W_{2k} $$
otherwise we denote $M_{2k} = F^{-1}(W_{2k}) \cap M_F$ and $W'_{2k} =  W_{2k} \setminus (Sing(W_{2k})\cup  S_{F_{\vert M_{2k}}})$. 
We put 
\begin{equationth} \label{filtrationSF}
W_{2k-1} = W_{2k-2}= Sing(W_{2k}) \cup S_{F_{\vert M_{2k}}} \cup A_{2k}, 
\end{equationth}
where  $A_{2k}$ is the smallest algebraic set which contains the set: 
$$B_{2k} = \left\{ x \in W'_{2k}  : 
\begin{matrix} 
 \text{ if $x\in W_{h}$ with $h>2k$ then   } \hfill \cr
\text{ the Whitney  $(b)$ condition fails at $x$ for the pair $(W'_{2k}, W_h)$}\\
\end{matrix}
\right\}.$$

Now, consider the filtration $(\mathcal{V})$ of $N_F$
$$(\mathcal{V}) : \qquad  N_F = V_{2n}  \supset V_{2n-1} \supset V_{2n-2} \supset \cdots \supset  V_{2k+1} \supset V_{2k} \supset \cdots 
\supset V_1 
\supset V_0 \supset \emptyset$$
where $V_i = \pi_F^{-1} (W_i)$ and $\pi_F$ is the canonical projection from $N_F$ to $\R^{2n}$, on the first $2n$ coordinates (see diagram (\ref{dia1})).

\medskip

Let $S'_{2i} = W_{2i} \setminus W_{2i-2}$. 
We claim that $F_{\vert F^{-1}(S'_{2i})}$ is proper. 
This is obvious if  $S'_{2i}$ is empty. 
If $S'_{2i}$ is not empty, suppose that there exists a sequence  $\{ x_l \}$ in $F^{-1}(S'_{2i})$ such that $F(x_l)$ 
 goes to a point  $a$ in $S'_{2i}$. 
We have to show that the sequence $\{ x_l \}$ does not go to  infinity. Since $S'_{2i} = W_{2i} \setminus W_{2i -2}$, 
where $W_{2i-2} = Sing(W_{2i}) \cup S_{F_{\vert M_{2i-2}}} \cup A_{2i}$, 
 we have $a \notin S_{F_{\vert M_{2i-2}}}$. 
If  $x_l$ tends to infinity then $a \in  S_{F_{\vert F^{-1}(S'_{2i})}} $, which is a  contradiction.

Let $X$ be a connected component of $\pi_F^{-1}(Z)$, where $Z \subseteq W_{2i} \setminus W_{2i-2}$. We have $X \subseteq V_{2i} \setminus V_{2i-2}$. We claim that either $X \subseteq Z \times \{ 0_{\R^p}\}$ or $X \cap (S_F \times \{ 0_{\R^p}\}) = \emptyset.$ Assume that there exist $x' \in X$ but $x' \notin Z \times \{ 0_{\R^p}\}$ and $x'' \in X \cap (S_F \times \{ 0_{\R^p}\})$. Then we have $x'' = (x, 0_{\R^p})$, where $x \in S_F$. There exists a curve $\gamma(t) = (\gamma_1(t), \gamma_2(t)) \subseteq X$ where $\gamma_1(t) \subseteq \R^n$ and $\gamma_2(t) \subseteq \R^p$,  such that $\gamma(0) = x'$  and $\gamma(1) = x''$. Let us call $u=\gamma(t_0)$ the first point at which $\gamma$ meets $S_F \times \{ 0_{\R^p}\}$. Thus, we have $\gamma_2(t) \neq 0$ whenever $t < t_0$ and $h_F^{-1}(\gamma(t))$ is in $M_{2i}$, for $t< t_0$.  Moreover, $F(h_F^{-1}(\gamma(t)))=\pi_F(\gamma(t))$ 
 tends to $\pi_F(u)$ and $h_F^{-1}(\gamma(t))$ tends to infinity as $t$ tends to $t_0$. Hence, $\pi_F(u) \in S_{F\vert M_{2i}} \subset W_{2i-2}$, so $u$ is in $V_{2i-2}$,  contradicting $u \in X\subset V_{2i}\setminus V_{2i-2}$.

\medskip

Let us show that $S_{2i}:=V_{2i}\setminus V_{2i-2}$ is a smooth manifold, for all $i$. Because $F_{ \vert F^{-1}(S'_{2i})}$ is proper, the restriction of $\pi_F$ to $\pi_F^{-1}(S'_{2i}) \setminus (S_F \times \{ 0_{\R^p}\}) = h_F (F^{-1}(S'_{2i}))$ is proper. Consequently, $\pi_F$ is a covering map on $S_{2i}$. This implies that $S_{2i}$ is a smooth manifold.

\medskip

Observe that in the case where $\overline{X}\cap (S_F\times \{0_{\R^p}\})$ is nonempty then it is included in $W_{2i-2} \times \{ 0_{\R^p}\}$, if $\dim X=2i$, since every point of $\overline{X}\cap (S_F\times \{0_{\R^p}\})$ is a point of  $S_{F\vert M_{2i}} \subseteq W_{2i-2}$. As $\pi_F$ is a covering map on $N_F \setminus (S_F \times \{0_{\R^p}\})$, this implies that  $S'_{2i} \times \{ 0_{\R^p}\}$ is open in $\pi_F^{-1}(S'_{2i})$. 

\medskip

Let us prove that the filtration $(\mathcal{V})$ defines  a Whitney stratification: at first, we prove that the stratification $(\mathcal{W})$ is a Whitney stratification. If the stratum $S'_{2i} = W_{2i} \setminus W_{2i-2}$ is not empty, then by (\ref{filtrationSF}), we have 
$$S'_{2i} = W_{2i} \setminus W_{2i-2} \subset W_{2i} \setminus A_{2i} \subset W_{2i} \setminus B_{2i}.$$
This shows that the stratification $(\mathcal{W})$ satisfies  Whitney conditions. 

We denote 
$$ \qquad \Sigma_{\mathcal{W}}: = \{ X' : X' \text{ is a connected component of } W_{2i} \setminus W_{2i -2}, 0 \leq i \leq n \},$$
$$\Sigma_{\mathcal{V}}: = \{ X : X \text{ is a connected component of } V_{2i} \setminus V_{2i -2}, 0 \leq i \leq n \}.$$ 
We now prove that if $X \in \Sigma_{\mathcal{V}}$ then $\pi_F(X) \in \Sigma_{\mathcal{W}}$.
 If $X \subseteq S_F \times \{ 0_{\R_p}\}$ then $\pi_{F_ { \vert X}}$ is the identity and thus $X$ belongs to $\Sigma_{\mathcal{W}}$. 
Otherwise, $X \subseteq N_F \setminus (S_F \times \{ 0_{\R^p}\}$. 
Assume that  $X \subseteq V_{2i} \setminus V_{2i-2}$. 
This implies that $X \cap \pi_F^{-1}(W_{2i-2}) = \emptyset$. 
This amounts to say that $\pi_F(X) \cap W_{2i-2} = \emptyset$. 
Thus $\pi_F(X) \subseteq W_{2i} \setminus W_{2i-2}$. 
Therefore, to show that $\pi_F(X) \in \Sigma_{\mathcal{W}}$, 
we have to check that $\pi_F(X)$ is open and closed in $W_{2i} \setminus W_{2i-2}$. 
As $\pi_F$ is a local diffeomorphism at any point $x$ of $X$, 
the set  $\pi_F(X)$ is a manifold of dimension $2i$, which is open in $S'_{2i}$. Let us show that it is closed in $S'_{2i}$.  
Take a sequence $y_m \subset  \pi_F(X)$ such that $y_m$ tends to $y \notin \pi_F(X)$. 
Let  $x_m \in X$ be such that $\pi_F(x_m) = y_m$. Since $\pi_F$ is proper, 
$x_m$ does not tend to infinity. 
Taking a subsequence if necessary, we can assume that $x_m$ is convergent. Denote its limit by $x$. 
As $\pi_F(x) = y \notin \pi_F(X)$, then the point $x$ cannot be in $X$ and thus belongs to $V_{2i-2}$ since $X$ is closed in $V_{2i} \setminus V_{2i-2}$. This implies that $y= \pi_F(x)\in W_{2i-2}$, as required. 

Let us consider a pair of strata $(X, Y)$ of the stratification $(\mathcal{V})$ such that $\overline{X} \cap Y \neq \emptyset$ and let us prove that $(X, Y)$ satisfies the Whitney $(b)$ condition. That is clear if $X, Y \subseteq S_F \times \{ 0_{\R^p} \}$. If none of them is included in $S_F \times \{ 0_{\R^p}\}$, then, as $\pi_F$ is a local diffeomorphism and Whiney $(b)$ condition is a ${\mathcal{C}}^1$ invariant, this is also clear. Therefore, we can assume that $X \cap (S_F \times \{ 0_{\R^p}\}) = \emptyset$ and $Y \subseteq S_F \times \{ 0_{\R^p}\}$ (if $X \subseteq S_F \times \{ 0_{\R^p}\}$, then $Y$ meets $S_F \times \{ 0_{\R^p}\}$ at the points of $\overline{X}$ and then $Y \subseteq( S_F \times \{ 0_{\R^p}\})$). Set for simplicity $Y: =Y'\times \{ 0_{\R^p}\}$.

 As $Y$ is  open in $\pi_F^{-1}(Y')$, there exists a subanalytic open subset $U'$ of $N_F$ such that 
$$\overline{U'} \cap \pi_F^{-1}(Y') = Y' \times \{ 0_{\R^p}\}.$$
Let $U'': = h_F^{-1}(\overline{U'} \cap N_F \setminus (S_F \times \{ 0_{\R^p}\}))$. We have 
$$U'' \cap F^{-1}(Y') = \emptyset$$
(see diagram (\ref{dia1})). Consequently, the function distance $d(F(x); Y')$ nowhere vanishes on $U''$. As $U''$ is a closed subset of $\R^{2n}$, by \L ojasiewicz inequalty, multiplying the $\psi_i$'s by a huge power of $\frac{1}{1 + {\vert x \vert}^2}$, we can assume that on $U''$, for every $i$
\begin{equationth} \label{*}
\psi_i(z_m) << d(F(z_m); Y')
\end{equationth}
for any sequence $z_m$ tending to infinity. 

Now, in order to check that Whitney $(b)$ condition holds, we take $x_m \in X$ and $y_m \in Y$ tending to $y \in \overline{Y} \cap \overline{X}$. Assume that $l = \lim \overline{x_my_m} $ and $\tau = \lim T_{x_m}X$ exist, we have to check that $l$ is included in $\tau$. 

For every $m$, $x_m$ belongs to $h_F(U_j)$ for some $j$. Extracting a subsequence if necessary, we may assume that it lies in the same $h_F(U_j)$. On $U_j$, $\pi_F$ is invertible and its inverse is 
$$\hat{\psi}(y) = (y, \tilde{\psi_1}(y), \ldots, \tilde{\psi_p}(y)),$$
where $\tilde{\psi_i}(y) = \psi_i \circ F^{-1}_{\vert U_j}$ (see section \ref{constructionNF}, see also Proposition 2.3 in \cite{Valette}).

Let $x_m = (x'_m; \tilde{\psi}(x'_m))$ and $y_m = (y'_m, 0_{\R^{p}})$, where $x'_m = \pi_F(x_m)$ and $y'_m= \pi_F(y_m)$, then $x_m - y_m = (x'_m - y'_m, \tilde{\psi}(x'_m)).$ We claim that 
\begin{equationth} \label{**}
\tilde{\psi}(x'_m) << \vert x'_m - y'_m \vert.
\end{equationth}

If $z_m = F^{-1}(x'_m)$ then $F(z_m) = x'_m$, so that by (\ref{*}), we have
$$\tilde{\psi_i}({x'}_m) << d({x'}_m; Y') \leq \vert {x'}_m - {y'}_m \vert,$$
showing (\ref{**}). 

On one hand, the sets  $\pi_F(X)$ and $\pi_F(Y)$ belong to $\Sigma_{\mathcal{W}}$, they satisfy the Whitney $(b)$ condition. As a matter of fact
$$\lim \frac{x'_m - y'_m}{\vert x'_m - y'_m \vert} = l' \subseteq \tau' = \lim T_{x'_m} \pi_F(X)$$
(extracting a sequence if necessary, we may assume $\frac{x'_m - y'_m}{\vert x'_m - y'_m \vert}$ is convergent).
We have 
 $$\frac{x_m - y_m}{\vert x_m - y_m \vert} = \frac{(x'_m - y'_m, \tilde{\psi}(x_m))}{\vert x_m - y_m \vert} \to (l',0) = l.$$
On the other hand, observe that 
$$d_x \pi_F^{-1} = (Id, \partial_x \tilde{\psi_1}, \dots, \partial_x \tilde{\psi_p}).$$
Multypling $\psi$ by a huge power of $\frac{1}{1+ {\vert x \vert}^2}$, we can assume that the first order partial derivatives of $\tilde{\psi}$ at $x'_m$ tend to zero  as $m$ goes to infinity. Then $T_{x_m}X$ tends to $\tau = \lim T_{x_m} \pi_F(X) \times \{ 0_{\R^p}\} = \tau' \times \{ 0_{\R^p}\}$. But since  $l' \in \tau'$, $l = (l',0) \in \tau = \tau' \times \{ 0_{\R^{p}}\}$.
\end{proof}
\goodbreak

We now generalize Theorem \ref{valette4}. Firstly we notice that a polynomial map $F_i : \C^n \to \C$ can be written
 $$F_i = \Sigma_j{F_i}_j$$
where $F_{ij}$ is the part of degree $d_j$ in $F_i$. Let $d_k$ the highest degree in $F_i$, the leading form $\hat F_i$ of $F_i$
is defined as
$$\hat{F_i} : = {F_i}_{k}.$$ 

\begin{theorem} \label{valettethuy2}
{\it Let $F : \C^n \rightarrow \C^n$ be a polynomial mapping with nowhere vanishing Jacobian. 
If {$\rang_{\C} {(D \hat {F_i})}_{i=1,\ldots,n} > n-2$}, where $\hat {F_i}$ is the leading form of $F_i$,  then the following conditions are equivalent:

\begin{enumerate}
\item[(1)] $F$ is non proper,
\item[(2)] $H_2(N_F) \neq 0,$
\item[(3)] $IH_2^{\overline{p}} (N_F) \neq 0$ for any (or some) perversity $\overline{p},$
\item[(4)] $IH_{2n-2, BM}^{\overline{p}} (N_F) \neq 0$, for any (or some) perversity $\overline{p}.$
\end{enumerate}
}
\end{theorem}

Before proving this theorem, we give here some necessary definitions and lemmas.

\begin{definition}
A {\bf semi-algebraic family of sets} (parametrized by $\R$) 
{\rm is a semi-algebraic set $A \subset \R^n \times \R$, the last variable being considered as parameter.}
\end{definition}
\begin{remark}
{ \rm The semi-algebraic set $A \subset \R^n \times \R$ will be considered as a family parametrized by $t \in \R$. We write $A_t$, for ``the fiber of $A$ at $t$'', {\it i.e.}:
$$ A_t : = \{ x \in \R^n : (x, t) \in A\}. $$}
\end{remark}
\begin{lemma} [\cite{Valette}] \label{valette5} 
Let $\beta$ be a $j$-cycle and let $A \subset \R^n \times \R$ be a compact semi-algebraic family of sets with $\vert \beta \vert \subset A_t$ for any $t$. Assume that  $\vert \beta \vert$ bounds a $(j+1)$-chain in each $A_t$, $t >0$ small enough. Then $ \beta $ bounds a chain in $A_0$.
\end{lemma}

\begin{definition}[\cite{Valette}] 

{\rm Given a subset $X \subset \R^{n}$, we define the {\bf ``tangent cone at infinity''}, called ``{\bf contour apparent \`a l'infini}''
in \cite{Thuy1} by: 
$$ C_{\infty}(X):=\{\lambda \in \S^{n-1}(0,1) \text{ such that } \exists \varphi : (t_0, t_0 + \varepsilon] \rightarrow X \text { semi-algebraic,}$$
$$ \qquad \qquad \qquad \underset{t \rightarrow t_0}{\lim} \varphi (t) = \infty, \underset{t \rightarrow t_0}{\lim}\frac{\varphi(t)}{\vert \varphi(t) \vert} = \lambda \}. $$
}
\end{definition}

\begin{lemma} [\cite{Thuy1}] \label{lemmathuy1}
{\it Let $F=(F_1, \ldots , F_n) : \R^n \to \R^n$ be a polynomial mapping and $V$ the zero locus of   $\hat{F}: = (\hat{F_1}, \ldots, \hat{F_n})$, where $\hat{F_i}$ is the leading form of  $F_i$ for $i = 1, \ldots , n$. If $X$ is a subset of $\R^n$ such that $F(X)$ is bounded, then $C_\infty(X)$ is a subset of $\S^{n-1}(0,1) \cap V.$}
\end{lemma}
\begin{proof}[Proof]
By definition, $C_\infty(X)$ is included in $\S^{n-1}(0,1)$. We prove now that  $C_\infty(X)$ is included in $V$. In fact, given $\lambda \in C_\infty(X)$, then there exists a semi-algebraic curve $ \gamma : \, (t_0, t_0 + \varepsilon] \rightarrow X $  such that  $\underset{t \rightarrow t_0}{\lim} \gamma (t) = \infty$ and $\underset{t \rightarrow t_0}{\lim}\frac{\gamma(t)}{\vert \gamma(t) \vert} = \lambda$. Then $\gamma(t)$ can be written as  
$\gamma(t) = \lambda t^m + \ldots$ and $\hat{F_i} = \hat{F_i}(\lambda)t^{md_i} + \ldots $ where $d_i$ is the homogeneous degree of $\hat{F_i}$.  
Since $F(X)$ is bounded, then $F_i$ cannot tend to infinity when $t$ tends to $t_0$, hence 
$ \hat{F_i}(\lambda) =0$  for all $i = 1, \ldots , n$.
\end{proof}

Let us prove now Theorem \ref{valettethuy2}. We use the idea and technique 
of the second and third authors in \cite{Valette}.

\begin{proof} [Proof of the Theorem \ref{valettethuy2}] 

 $(4) \Leftrightarrow (3)$ : By Goresky-MacPherson Poincar\'e  duality Theorem, we have
 $$IH_2^{\overline{p}}(N_F) = IH_{2n-2,BM}^{\overline{q}}(N_F),$$
where $\overline{q}$ is the complementary perversity of $\overline{p}$. Since  $IH_2^{\overline{p}}(N_F) \neq 0$ for all perversities  $\overline{p}$, then  $IH_{2n-2}^{\overline{q}}(N_F) \neq 0$, for all perversities $\overline{q}.$

$(3) \Rightarrow (1), (3) \Rightarrow (2) $ : If $F$ is proper then the sets   $S_F$ and $K_0(F)$ are empty. 
So $Sing(N_F)$ is empty and  $N_F$ is homeomorphic to  $\R^{2n}$. 
It implies  that $H_2(N_F) = 0$ and $IH_2^{\overline{p}}(N_F) = 0.$

 $(1) \Rightarrow (2), (1) \Rightarrow (3)$ : Assume that $F$ is not proper. That means that there exists a complex Puiseux arc  
$\gamma : D(0, \eta) \rightarrow \R^{2n}$, $ \gamma = u z^{\alpha} + \ldots, $
(with $\alpha$ negative integer and $u$
 is an unit vector of $\R^{2n}$) tending to infinity in such a way that $F(\gamma)$ converges to a generic point $x_0 \in S_F$. Let $\delta$ be an oriented triangle in $\R^{2n}$ whose barycenter is the origin. Then, as the mapping $h_F \circ \gamma$ (where $h_F = (F, \psi_1, \ldots,\psi_p))$ extends continuously at $0$, it provides a singular $2$-simplex in $N_F$ that we will denote by $c$.

\noindent Since $\codim_{\R} S_F = 2$, then
$$ 0 = \dim_{\R} \{ x_0 \} = \dim_{\R} ( (S_F \times \{ 0_{\R^{p}} \}) \cap \vert c \vert) \leq 2 -2 + p_2,$$
because $p_2 = 0$ for any perversity $\overline{p}$. So the simplex $c$ is $(\overline{0},2)$-allowable for any perversity $\overline{p}$. 

\noindent The support of $\partial c$ lies in $N_F \setminus S_F \times \{0_{\R^p}\}$. By definition of $N_F$, we have 
$N_F \setminus S_F \times \{ 0_{\R^{p}} \} \simeq \R^{2n} $. Since $H_1(\R^{2n}) = 0$, the chain $\partial c $ bounds 
a singular chain  
$e \in C^2(N_F \setminus S_F \times \{0_{\R^{p}}\})$. 
So $\sigma = c- e$ is a $(\overline{p},2)$-allowable cycle of  $N_F$. 

We claim that $\sigma$ may not bound a  $3$-chain in $N_F$.
Assume otherwise, {\it i.e.} assume that there is a  chain $\tau \in C_3(N_F)$, satisfying $\partial \tau=\sigma$. Let 
$$A:= h_F^{-1}(\vert \sigma \vert \cap (N_F \setminus (S_F \times \{0_{\R^{p}}\}))),$$ 
$$B:= h_F^{-1}(\vert \tau \vert \cap (N_F \setminus (S_F \times \{0_{\R^{p}}\}))).$$ 

By definition, $C_\infty(A)$ and $C_{\infty}(B)$ are subsets of $\S^{2n-1}(0,1)$. 
 Observe that, in a neighborhood of infinity,  $A$ coincides with the support of the  Puiseux arc $\gamma$.  The set $C_\infty(A)$ is equal to $\S^1.a$ (denoting the orbit of $a \in \C^n$ under the action of $\S^1$ on $\C^n$, $(e^{i\eta},z) \mapsto  e^{i\eta}z$). Let $V$ be the zero locus of the leading forms  $\hat{F}: = (\hat{F_1}, \ldots, \hat {F_n})$.  Since $F(A)$ and $F(B)$ are bounded, by  Lemma \ref{lemmathuy1}, $C_\infty(A)$ and $C_{\infty}(B)$ are subsets of $V \cap \S^{2n -1}(0,1).$

For $R$ large enough, the sphere $\S^{2n-1}(0, R)$ with center 0 and radius $R$ in $\R^{2n}$ is transverse to $A$ and $B$ (at regular points). Let 
$$\sigma_R : = \S^{2n-1}(0 , R) \cap A, \qquad \tau_R : = \S^{2n-1}(0, R) \cap B.$$
After a triangulation, the intersection $\sigma_R$ is a chain bounding the chain $\tau_R$. 

Consider a semi-algebraic strong deformation retraction 
$\rho : W \times [0;1] \rightarrow \S^1.a$, where $W$ is a neighborhood of  $\S^1.a$ in $\S^{2n-1}(0,1)$ onto $\S^1.a$. 

Considering $R$ as a parameter, we have the following semi-algebraic families of  chains: 

1) $\tilde{\sigma _R} :=\frac{\sigma _R}{R},$
for $R$ large enough, then $\tilde{\sigma _R}$ is contained in $W$,

2)  $\sigma{'}_{R} = \rho_1(\tilde{\sigma_R})$, where $\rho_1(x) : = \rho(x, 1), \qquad x \in W$,

3) $\theta_{R} = \rho(\tilde{\sigma_R})$, we have 
 $\partial \theta_R  = \sigma'_R - \tilde{\sigma_R} ,$

4) $\theta{'}_{R} = \tau_R + \theta_R$, we have 
 $\partial \theta'_R  =  \sigma'_R.$

As, near infinity, $\sigma_R$ coincides with the intersection of the support of the arc $\gamma$ with $\S^{2n -1}(0, R)$, for $R$ large enough the class of $\sigma'_R$ in $\S^1.a$ is nonzero.

Let $r = 1/R$, consider $r$ as a parameter, and let $\{ \tilde{\sigma_r} \}$, $\{ \sigma'_r \}$, $\{ \theta_r \}$ as well as $\{ \theta'_r \}$ the corresponding  semi-algebraic families of  chains. 

Denote by $E_r \subset \R^{2n} \times \R$ the closure of $|\theta_r|$, and set $E_0 : = (\R^{2n} \times \{ 0 \}) \cap E$. Since the strong deformation retraction  $\rho$ is the identity on $C_{\infty}(A) \times [0,1]$, we see that  
$$E_0 \subset \rho(C_{\infty}(A) \times [0,1]) = \S^1. a \subset V \cap \S^{2n-1}(0,1).$$

Denote $E'_r \subset \R^{2n} \times \R$ the closure of $|\theta'_r|$, and set  $E'_0 : = (\R^{2n} \times \{ 0 \}) \cap E'$. Since $A$ bounds $B$, so $C_{\infty}(A)$ is contained in $C_{\infty}(B)$. We have 
$$E'_0 \subset E_0 \cup C_{\infty}(B) \subset V \cap \S^{2n-1}(0,1).$$

The class of $\sigma'_r$ in $\S^1.a$ is, up to a product with a nonzero constant, equal to the generator of $\S^1.a$. Therefore, since $\sigma'_r$ bounds the chain $\theta'_r$, the cycle $\S^1.a$ must bound a chain in $|\theta'_r|$ as well. By Lemma \ref{valette5}, this implies that $\S^1.a$ bounds a chain in $E'_0$ which is included in $V \cap \S^{2n-1}(0,1)$.

 The set $V$ is a projective variety which is an union of cones in $\R^{2n}$. Since $\rang_{\C} {(D \hat {F_1})}_{i=1,\ldots,n}  >n-2$ then $ \corang_{\C} {(D \hat {F_1})}_{i=1,\ldots,n} = \dim _{\C} V \leq 1$, so $\dim_{\R} V \leq 2 $ and  $\dim_{\R} V \cap \S^{2n-1}(0,1) \leq 1.$
The cycle $\S^1 . a$ thus bounds a chain in $E'_0\subseteq V \cap \S^{2n-1}(0,1)$, which is a finite union of circles. A contradiction.
\end{proof}

%
%
%
We have the following corollary
\begin{corollary}
{\it Let $F = (F_1, \ldots,  F_n): \C^n \rightarrow \C^n$ be a polynomial mapping with nowhere vanishing
Jacobian and such that $\rang_{\C} {(D \hat {F_i})}_{i=1,\ldots,n} > n-2$, where $\hat {F_i}$ is the leading form of $F_i$. The following conditions are equivalent:
\begin{enumerate}
\item[1)] $F$ is nonproper,
\item[(2)] $H_{\infty}^2(\Reg (N_F^R)) \neq 0,$
\item[(3)] $H_{\infty}^{n-2}(\Reg (N_F^R)) \neq 0.$
\end{enumerate}
where $N_F^R:=N_F\cap \bar B (0,R)$, which $R$ is large enough.}
\end{corollary}

The proof is similar as the one  in \cite{Valette}.


\bibliographystyle{plain}

\end{document}